\documentclass{amsart}
\usepackage{amssymb}

    \usepackage[
    top    = 2.0cm,
    bottom = 1.8cm,
    left   = 2.54cm,
    right  = 2.54cm]    {geometry}

\usepackage{stmaryrd}
\usepackage{mathrsfs}
\usepackage{leftidx}
\usepackage{enumerate}

\newtheorem{theorem}{Theorem}[section]
\newtheorem{lemma}[theorem]{Lemma}
\newtheorem{corollary}[theorem]{Corollary}
\newtheorem{proposition}[theorem]{Proposition}
\theoremstyle{definition}

\theoremstyle{remark}

\numberwithin{equation}{section}

\newcommand{\FF}{{\mathbb{F}}}

\newcommand{\bN}{{\mathbf{N}}}

\newcommand{\Aut}{{\operatorname{Aut}}}

\newcommand{\Syl}{{\operatorname{Syl}}}

\newcommand{\GL}{{\operatorname{GL}}}

\newcommand{\GF}{\mbox{GF}}

\newcommand{\Out}{{\operatorname{Out}}}

\newcommand{\ug}{\ \raisebox{-.3em}{$\stackrel{\scriptstyle \leq}
{\scriptstyle \sim}$} \ }
\newcommand{\n}{{\mbox{\rm I$\!$N}}}
\newcommand{\X}{{\mbox{$\setminus$}\mbox{$\!\!\!/$}}}

\begin{document}

\title{Abelian quotients and orbit sizes of linear groups}

\author{Thomas Michael Keller}
\address{Department of Mathematics, Texas State University, 601 University Drive, San Marcos, TX 78666, USA.}

\author{YONG YANG}
\address{Key Laboratory of Group and Graph Theories and Applications, Chongqing University of Arts and Sciences, Chongqing, China and Department of Mathematics, Texas State University, 601 University Drive, San Marcos, TX 78666, USA.}

\makeatletter
\email{keller@txstate.edu, yang@txstate.edu}
\makeatother

\date{}


\begin{abstract}
Let $G$ be a finite group, and let $V$ be a completely reducible faithful finite $G$-module (i.e., $G\leq \GL(V)$, where $V$ is a finite vector space which is a direct sum of irreducible $G$-submodules). It has been known for a
long time that if $G$ is abelian, then $G$ has a regular orbit on $V$. In this paper we show that $G$ has an
orbit of size at least $|G/G'|$ on $V$.
This generalizes earlier work of the authors, where the same bound was proved
under the additional hypothesis that $G$ is solvable. For completely reducible modules it also strengthens
the 1989 result $|G/G'|<|V|$ by Aschbacher and Guralnick.

\end{abstract}

\maketitle
\Large
\section{Introduction} \label{sec:introduction8}

This paper is a sequel to \cite{keller-yang}. In that paper we proved that if $G$ is a finite solvable group and
$V$ is a completely reducible faithful $G$-module (possibly of mixed characteristic), then $G$ has an orbit of size
at least $|G/G'|$ on $V$, where $G'=[G,G]$ is the commutator subgroup of $G$. Following the standard reference \cite{manz-wolf} we say that $V$ is allowed to be
of {\it mixed characteristic} if $V=V_1 \oplus \dots \oplus V_t$ where each $V_i$ is an irreducible $G$-module over a finite field $\FF_i$ of characteristic $p_i$. This is just a slight generalization of the usual notion of completely reducible.\\

In the case that $G'=1$, i.e., $G$ is abelian, it has been known for sometime that the action of $G$ on $V$ has a regular orbit. This is a special case of a natural and intuitive generalization (see Theorem \ref{thm1} for a precise statement) that we expected to hold true for arbitrary groups. The proof, however, even just for solvable groups, turned out to be quite difficult and in \cite{keller-yang} we left
the problem for arbitrary finite groups as a conjecture.\\

In this paper we are able to prove this conjecture.  Our main result is the following.

\begin{theorem}\label{thm1}
Let $G$ be a finite group and $V$ a finite faithful completely reducible $G$-module,
possibly of mixed characteristic. Let $M$ be the largest orbit size in the action of $G$ on $V$.
Then
\[|G/G'|\leq M.\]
\end{theorem}

The reader should compare this result to a 1989 result of Aschbacher and Guralnick \cite{aschbacher-guralnick}. They proved that if $V$ is a faithful $G$-module of characteristic $p$ for a finite group $G$ satisfying $\text{O}_p(G)=1$, then $|G/G'|<|V|$. Hence for completely reducible $V$, Theorem \ref{thm1} strengthens this bound.
Thus our result is a nice combination of the classical regular orbit result for abelian groups and the Aschbacher-Guralnick result for completely reducible modules.\\ 

Results on orbit sizes, particularly on the existence of large orbits, have been fundamental to solving problems in the modular and ordinary character theory of finite groups. Our main result in this paper fits in this category, and it has already been used in ~\cite{YY23} and ~\cite{Sambale}.

In ~\cite{YY23}, Qian and the second author use our result to prove a strengthening of the well-known result that the nilpotent length (or Fitting length) of a finite solvable group is bounded above by the number of its irreducible complex character degrees. In ~\cite{Sambale}, Sambale uses our result to prove a special case of the famous Brauer's $k(B)$-conjecture; namely he proves that it is true if $B$ is a block with abelian defect group $D$ and inertial quotient $T$ such that $|T'|\leq 4$. (This is a dual result to one he proved earlier where the hypothesis is that $T$ contains an abelian subgroup of index 4 instead of the condition that $|T'|\leq 4$.)


We now give an outline of the main ideas guiding our proof. In the solvable case, the main idea essentially was to show that $|G:G'|$ is bounded above by $|F(G):F(G)'|$, where $F(G)$ is the Fitting subgroup of $G$ (see Theorem ~\ref{thm2}). This allowed a reduction to nilpotent groups. Nilpotent groups could be further reduced to those of nilpotence class 2, which turns out to be the hardest case.
To achieve the generalization to arbitrary finite groups in the current paper, we try to emulate the
approach used in the solvable case. Here the generalized Fitting subgroup $F^*(G)$ plays the role of the
Fitting subgroup. Now $F^*(G)=F(G)E(G)$, where $E(G)$ is the layer of $G$, and one can think
of some part of $G/G'$ acting faithfully on $F(G)$ and the other part acting faithfully on $E(G)$.
We bound the first part by $|F(G):F(G)'|$ and the second part by the order of some abelian subgroup of $E(G)/Z(E(G))$. This portion combined with $F(G)$ yields a nilpotent subgroup $T$ of $G$ such that
$|G/G'|\leq |T/T'|$.  We then use the result for solvable groups to get the desired conclusion.
To that end, in Section 2 we prove that if $C$ is the centralizer in $G$ of the factor group $E(G)$ mod
its center, then $|C:C'|\leq |F(G):F(G)'|$.  This takes care of the above-mentioned first part
of $|G/G'|$. We deal with the second part in Section 4, with Section 3 providing
some number theoretical results that we need.
A major ingredient here is a result on finite simple groups which guarantees the existence of abelian subgroups of sufficiently large order. Namely, we need their order to be larger than the
order of the abelian quotient of the outer automorphism group of the group (see Section 4). The proof of this result goes through CFSG and is admittedly quite tedious, but
its usefulness makes up for that. The proof in Section 4 also uses a refinement of a result by Feit \cite{feit} on large Zsigmondy prime divisors (see Section 3). Finally, in Section 5 we put the two pieces together to prove the main result.\\

\section{Abelian quotients and the Fitting subgroup} \label{sec:1}

We need the following generalization of \cite[Theorem 3]{aschbacher-guralnick}.

\begin{lemma}\label{lem1}
Let $G$ be a finite group and $V$ a finite, faithful $G$-module,
possibly of mixed characteristic. Suppose that $V=F^*(GV)$. Then
\[|G:G'|<|V|.\]
\end{lemma}

\begin{proof}
Let $G$ and $V$ be a counterexample such that $|GV|$ is minimal. Now
assume that $G$ acts trivially on some Sylow $p$-subgroup $X$ of $V$
for some prime $p$ dividing $|V|$. Then $V=X\times Y$ for a Hall $p'$-subgroup $Y$
of $V$, and $G$ acts faithfully on $Y$ such that
\[Y=F^*(GY).\]

Since $|GY|<|GV|$, then $G$ and $Y$ do not form a counterexample to our statement,
hence we conclude that
\[|G:G'|<|Y| < |V|\]
and we are done. Thus we may assume that $G$ acts nontrivially on every
Sylow subgroup of $V$ (for every prime dividing $|V|$). \\

Write $V=P_1\oplus\cdots\oplus P_l$ for some
$l\in\n$, where $P_i\in\Syl_{p_i}(V)$ for the distinct primes $p_i$
dividing $|V|$.\\

We now use an inductive process to define an ascending chain of
$G$-submodules $V_i$ of $V$ ($i=0,\dots, n$) (where
$n$ is a positive integer determined by the process), and simultaneously
we define a descending chain of subgroups $G_i$ of $G$ ($i=0,\dots, n$) and again
simultaneously $G$-modules $W_i$ ($i=0,\dots, n$) and integers $0=t_0<t_1<\cdots <t_l=n$
such that
\[V_{t_{i+1}}/V_{t_i}\cong
W_{t_i+1}\oplus\cdots\oplus W_{t_{i+1}}\cong P_{i+1}\] (where "$\cong$"
just stands for isomorphism as groups, not as $G$-modules) for
$i=0,\ldots,l-1$.\\

That is, we define
\[0=V_0<V_1<\cdots<V_n=V\]
for some $n\in\n$ and $G_i$ ($i=1,\ldots,n$) as follows:\\

Let $G_0=G$, $V_0=0$, and $W_0=0$, and $t_0=0$. Now suppose that the integer
$t_i$ is already defined for some $i\in\{0,\dots, l-1\}$, and that for some
integer $j\geq 0$ also $G_{t_i+j},V_{t_i+j}, W_{t_i+j}$ are already defined and that
$G_{t_i+j} = C_G(W_1\oplus\dots\oplus W_{t_i+j})$ \\

(+) First consider the case that $p_{i+1}$ divides $|V/V_{t_i+j}|$.\\

Then let $V_{t_i+j+1}$ be such that $V_{t_i+j}<V_{t_i+j+1}\leq V$ and
$W_{t_i+j+1}=V_{t_i+j+1}/V_{t_i+j}$ is an irreducible $G$-module over $\GF(p_{i+1})$.
Let $G_{t_i+j+1}=C_{G_{t_i+j}}(W_{i+1})\leq G_{t_i+j}$. Observe that
$G_{t_i+j+1}=C_G(W_1\oplus\cdots\oplus W_{t_i+j+1})\unlhd G$.\\

So in this case we have defined $G_{t_i+j+1},V_{t_i+j+1}$, and $W_{t_i+j+1}$ in the desired way.\\

Now we consider the other case, namely that $p_{i+1}$ does not divide $|V/V_{t_i+j}|$.\\
In this case we have completed the prime $p_{i+1}$ and put $t_{i+1}=t_i+j$. Note that
then $G_{t_{i+1}} = C_G(W_1\oplus\dots\oplus W_{t_{i+1}})$.
If $i+1=l$, then we have defined all the modules and subgroups we need; so we
let $n=t_i+j$ and stop the process. If $i+1<l$, then we move to the next prime $p_{i+2}$
with the next iteration (i.e., we go to (+) again).\\

Next we claim that for $G_n=C_G(W_1\oplus\cdots\oplus W_n)$ we
have
\[(*)\quad G_n=1.\]

Now observe that
$\bigcap\limits_{i=1}^lC_G(P_i)=1$, and so
\[G_n\ug\X_{i=1}^lG_n/C_{G_n}(P_i).\]
By \cite[Corollary 5.3.3]{gorenstein} we know that
$G_n/C_{G_n}(P_i)$ is a $p_i$-group. Thus $G_n$ is nilpotent of
order having only prime divisors from $V$. Moreover, if $x\in G_n$
is of $p_i$-power order for some $i\in\{1,\ldots,l\}$, then $x$ acts
trivially on $P_j$ for all $j\not=i$, because $xC_{G_n}(P_j)\in
G_n/C_{G_n}(P_j)$ and the latter is a $p_j$-group. This shows that
$G_nV$ is nilpotent, and as $G_nV\unlhd GV$, by our hypothesis we
conclude that
\[G_nV\leq F^*(GV)=V,\]
so $G_n=1$ follows and $(*)$ is proved.\\

Next we fix $i\in\{0,\ldots,n-1\}$ and consider the action of
$H_i:=G_i/G_{i+1}$ on $W_{i+1}$. Let $r$ be the characteristic of
$W_{i+1}$. Since $W_{i+1}$ is an irreducible faithful
$G/C_G(W_{i+1})$-module, we have that
\[O_r(G/C_G(W_{i+1}))=1.\]

Since \[H_i=G_i/C_{G_i}(W_{i+1})=G_i/(C_G(W_{i+1})\cap G_i)\cong
G_iC_G(W_{i+1})/C_G(W_{i+1})\unlhd G/C_G(W_{i+1}),\] we see that
also $O_r(H_i)=1$.\\

Hence by \cite[Theorem 3]{aschbacher-guralnick} we
conclude that $|H_i:H_i'|<|W_{i+1}|$.\\

Now by \cite[Lemma 2.1]{keller-yang} and since $G_n=1$, we have that
\[|G:G'|\leq\prod_{i=0}^{n-1}|H_i:H_i'|<\prod_{i=0}^{n-1}|W_{i+1}|=|V|.\]
This concludes the proof of the lemma.
\end{proof}

\begin{theorem}\label{thm2}
Let $G$ be a finite group with $F^*(G)=F(G)$. Then
\[|G:G'|\leq|F(G):F(G)'|.\]
\end{theorem}
\begin{proof}
We let $G$ be a counterexample of minimal order. Clearly $|G|>1$. Put
$N=\Phi(F(G))$ (where $\Phi(U)$ denotes the Frattini subgroup of a group $U$). By \cite[Lemma 3.16(a)]{gls} we have
\[(1)\quad F^*(G/N)=F^*(G)/N=F(G)/N=F(G/N).\]
So if $N>1$, we may apply the induction hypothesis which yields that
\[|G:G'N|=|(G/N):(G/N)'|\leq|F(G/N):F(G/N)'|=|F(G)/N:(F(G)/N)'|,\]
where the last equality follows from (1). Now as $F(G)/N$ is
abelian, it follows that
\[(2)\quad |G:G'N|\leq|F(G)/N|.\]

Next observe that $F(G)'\leq N\cap G'$. Hence we have
\[(3)\quad\frac{|N|}{|F(G)'|}\geq\frac{|N|}{|N\cap G'|}=|N:N\cap
G'|=|NG':G'|.\]

Putting (2) and (3) together gives
\[|G:G'|=|G:G'N||G'N:G'|\leq|F(G)/N|\cdot\frac{|N|}{|F(G)'|}=|F(G):F(G)'|,\]
and we are done in the case that $N>1$.\\

Hence from now on we may assume that $N=1$. Also, $F(G)$ is abelian of squarefree exponent, and
$C_G(F(G))=F(G)$, so $G/F(G)$ acts faithfully on $F(G)$. We have to
show that $|G:G'|\leq |F(G)|.$\\

Now observe that $F(G')=F(G)\cap G'$. Let $F(G)\leq K\unlhd G$ be
such that $K/F(G)$ is the kernel of the action of $G/F(G)$ on
$F(G')$. Our next goal is to show that $K=F(G)$.\\

Write
\[V_0=V=F(G) \ \mbox{ and }V_1=F(G')\]
(reading $V$ as a $G/F(G)$-module of possibly mixed
characteristic).\\

Note that $H:=G/F(G)$ acts trivially on $V/V_1$, because if $g\in G$
and $x\in F(G)$, then $x^g=x[x,g]$ where $[x,g]\in F(G)\cap
G'=F(G')$, so that
\[(xF(G'))^g=xF(G').\]

Thus $H$ acts faithfully on $V$ and acts trivially on both $V/V_1$ and
$V_1$. Let $\Pi$ be the set of common prime divisors of $|V/V_1|$
and $|V_1|$. If $\Pi=\emptyset$, then it is clear that $K=F(G)$.\\

Let $\Pi\not=\emptyset$. If $p\in\Pi$ and $P\in\Syl_p(V)$, then
either $K\leq C_G(P)$ or $K/C_K(P)$ is a $p$-group by
\cite[Corollary 5.3.3]{gorenstein}. Now let $V=X\oplus Y$, where $X$
is the Hall $\Pi$-subgroup of $V$, and $Y$ is the Hall
$\Pi'$-subgroup of $V$. Clearly $K$ acts trivially on $Y$, so as
$K/F(G)$ acts faithfully on $V$, we see that $K/F(G)$ acts
faithfully on $X$. Let $p_i$ ($i=1,\ldots,n$ for some $n\in\n$) be
the distinct prime numbers in $\Pi$, so that
\[X=P_1\times\ldots\times P_l\]
for $P_i\in\Syl_{p_i} (X)$ for $i=1,\ldots,l$. Write $\bar{K}=K/F(G)$.
Then
\[\bar{K}\ug\X_{i=1}^l\bar{K}/C_{\bar{K}}(P_i),\]
and from the above we know that $\bar{K}/C_{\bar{K}}(P_i)$ is a
$p_i$-group for all $i$. This shows that $\bar{K}$ is nilpotent of
order divisible only by primes from $\Pi$. Moreover, any
$p_i$-element of $\bar{K}$ can only act nontrivially on $P_i$ and
must act trivially on $P_j$ for all $j\not=i$, as
$\bar{K}/C_{\bar{K}}(P_j)$ is a $p_j$-group. This shows that
$\bar{K}F(G)$ is nilpotent. So if $\bar{K}>1$, then we may assume
that $\bar{Q_1}\in\Syl_{p_1}(\bar{K})$ is nontrivial, and then
$\bar{Q_1}\unlhd G/F(G)$. If $Q_1\leq G$ is the inverse image of
$\bar{Q_1}$ in $G$, then we easily see that $F(G)<Q_1$, and $Q_1$ is
a nilpotent normal subgroup of $G$ (since $\bar{Q_1}$ and thus also
the Sylow $p_1$-subgroup of $Q_1$ centralizes the Hall
$p_1'$-subgroup of $F(G)$), contradicting the definition of
$F(G)$.\\

Thus $\bar{K}=1$, and hence $K=F(G)$, as desired. Since $K=F(G)$, we
now know that $G/F(G)$ acts faithfully on $F(G')$, and $F(G')$ can
be viewed as a faithful $G/F(G)$-module of possibly mixed
characteristic. Moreover, if we look at the semidirect product $H$
of $G/F(G)$ and $F(G')$ with respect to the natural action, then we claim that
\[F(G')=F^*(H).\]
To prove the claim, note that clearly $F^*(H)=F(H)$, and so it suffices to show that
$F(H)=F(G')$. We prove this by contradiction. Assume that $F(H)>F(G')$. Then
$F(H)/F(G')\cong G/F(G)$ contains a normal $p$-subgroup
$\bar{S}>1$ for some prime $p$ such that $\bar{S}$ acts
trivially on the Hall $p'$-subgroup of $F(G')$ and also - since $G$
acts trivially on $F(G)/F(G')$ - on the Hall $p'$-subgroup of
$F(G)/F(G')$. Hence (again by \cite[Corollary 5.3.3]{gorenstein})
$\bar{S}$ acts trivially on the Hall $p'$-subgroup of $F(G)$,
and thus if $S\leq G$ with $F(G)\leq S$ is the inverse image of
$\bar{S}$ in $G$, then $S\unlhd G$ is nilpotent, contradicting
$F(G)<S$.\\

So indeed $F(G')=F^*(H)$, and hence we may apply Lemma \ref{lem1} to $H$. This yields
\[|G:G'F(G)|=|H:H'|\leq|F(G')|.\]
Thus altogether
\begin{eqnarray*}
|G:G'|&=&|G:G'F(G)||G'F(G):G'|\\
     &\leq &|F(G')|\cdot|F(G):(F(G)\cap G')|\\
     &=&|F(G')|\cdot|F(G):F(G')|\\
     &=&|F(G)|,
\end{eqnarray*}
which completes the proof of the theorem.
\end{proof}

\begin{corollary}\label{cor3}
Let $G$ be a finite group and $C=C_G(E(G)/Z(E(G)))$. Then $C=C_G(E(G))$ and
\[|C:C'|\leq |F(G):F(G)'|.\]
\end{corollary}
\begin{proof}
First note that since $E(G)$ is perfect, from \cite[Lemma 3.8]{gls} it immediately follows that
$C=C_G(E(G))$. With this we observe that
$F^*(C)=F(C)=F(G)$, as is well-known (see e.g.
\cite[Lemma 3.13(i)]{gls}). Therefore by Theorem \ref{thm2} we conclude that
\[|C:C'|\leq|F(C):F(C)'|=|F(G):F(G)'|,\]
and we are done.
\end{proof}

\section{Number Theory Results} \label{sec:3}

Let $a$ and $m$ be integers greater than 1. A Zsigmondy prime for $(a,m)$ is a prime $l$ such that $l \mid a^m-1$ but $l \nmid a^i-1$ for $1 \leq i \leq m-1$. A well-known theorem of Zsigmondy asserts that Zsigmondy primes exist except if $(a,m)=(2,6)$ or $m=2$ and $a=2^k-1$. If a Zsigmondy prime $l$ exists for $(a,m)$, we will say that there exists a Zsigmondy prime $l$ for $a^m-1$. If a Zsigmondy prime $l$ exists for $(a,2m)$, since $a^{2m}-1=(a^m-1)(a^m+1)$ and $l \nmid a^m-1$, we have that $l \mid a^m+1$. Thus for convenience, we will say that there exists a Zsigmondy prime $l$ for $a^m+1$.

Observe that if $l$ is a Zsigmondy prime for $(a,m)$, then $a$ has order $m$ modulo $l$ and so $l \equiv 1 \pmod{m}$. Thus $l \geq m+1$. A prime $l$ is a large Zsigmondy prime for $(a,m)$ if $l$ is a Zsigmondy prime for $(a,m)$ and either $l \geq 2m+1$ or $l^2 \mid a^m-1$. Walter Feit ~\cite[Theorem A]{feit} introduced the notion of large Zsigmondy prime and proved that large Zsigmondy primes exist except for a small list of exceptional cases. In this paper, we will use Feit's result along with some variations of Feit's result. Also, note that ~\cite[Theorem 2.2]{NP}
has some more information on large Zsigmondy primes.

\begin{lemma}\label{feitcor}
If $q$ is a prime power and $m$ is an integer greater than $2$, then there exists a large Zsigmondy prime $l$ for $q^m-1$ except in the following cases.
\begin{itemize}
\item $q=2$ and $m=4,6,10,12,18$.
\item $q=3$ and $m=4,6$.
\item $q=5$ and $m=6$.
\end{itemize}
\end{lemma}
\begin{proof}
This follows immediately from ~\cite[Theorem A]{feit}.
\end{proof}

\begin{lemma}\label{feitcorprime}
If $q$ is a prime power and $m$ is an integer greater than $2$, then there exists a prime $l \mid q^m-1$ such that either $l \geq 2m+1$ or $l^2 \mid q^m-1$ (and thus $l^2 > 2m+1$) except in the following cases.
\begin{itemize}
\item $q=2$ and $m=4,6,12$.
\item $q=3$ and $m=4$.
\end{itemize}
\end{lemma}
\begin{proof}
This follows immediately from ~\cite[Theorem A]{feit}. We observe that if $l^2 \mid q^m-1$, then since $l \geq m+1$, we have that $l^2 \geq (m+1)^2 = m^2+2m+1 > 2m+1$.
\end{proof}

\begin{lemma}\label{largercorprime}
If $q$ is a prime power and $m$ is an integer greater than $2$, then there exists a prime $l \mid q^m-1$ such that either $l \geq 3m+1$ or $l^2 \mid q^m-1$ (and thus $l^2 > 3m+1$) except in the following cases.
\begin{itemize}
\item $q=2$ and $m=3,4,6,8,12,20$.
\item $q=3$ and $m=4$ or $6$.
\item $q=4$ and $m=6$.
\end{itemize}
\end{lemma}
\begin{proof}
This follows from the results in ~\cite{Glasby}. We observe that if $l^2 \mid q^m-1$, then since $l \geq m+1$, we have that $l^2 \geq (m+1)^2 = m^2+2m+1 > 3m+1$.
\end{proof}

\section{Outer automorphism groups of simple groups} \label{section4}

We first give the general idea with respect to the proof of the following result. For most finite simple groups of Lie type, we use the results in Section 3 to find two prime divisors $L_1$ and $L_2$ of $|G|$, which are coprime, such that there exist subgroups $H_1$, $H_2$ with $|H_i| = L_i$ satisfying the conditions required. In particular, $L_1$, $L_2 \geq |\Out(G)|$ which is enough if $\Out(G)$ is nontrivial. There are some exceptional cases when either the rank or the size of the finite field is small, where one cannot find suitable Zsigmondy primes. We may deal with those exceptional cases by checking the bounds by direct calculation or computationally with GAP ~\cite{GAP}.

\begin{proposition}\label{prop1}
Let $G$ be a finite non-abelian simple group. Then there exist two abelian subgroups $H_1$ and $H_2$ of $G$ where $\gcd (|H_1|,|H_2|)=1$, such that for any subgroup $K$ of $\Out(G)$, we have $2 \sqrt{|K|} \leq |H_i|$ and ${|K/K'|} \leq |H_i|$ for $i=1,2$.
\end{proposition}
\begin{proof}
We now go through the Classification of Finite Simple Groups.

\begin{enumerate}[I]
\item Let $G$ be one of the alternating groups $A_n, n \geq 5$. We know that $|\Out(A_n)| = 2$ except when $n=6$ and $|\Out(A_6)| = 4$. If $|\Out(A_n)| = 2$, then the result is clear since $4 \mid |A_n|$ and $5 \mid |A_n|$ when $n \geq 5$. When $n=6$, $|\Out(A_6)| = 4$ and $|A_6|= 8\cdot9\cdot5$ and the result is also clear.\\

\item Let $G$ be one of the sporadic groups or Tits group. Thus $|\Out(G)| \leq 2$ and the result can be checked by ~\cite{CFSG}.\\

For simple groups of Lie type, we go through various families of groups of Lie type. To illustrate the method, we will give a detailed analysis for $A_n(q)$ and $\leftidx{^2} A_n(q^2)$, and show how to use the results in Section 3 to handle most of the cases. For those finitely many exceptional ``small" cases left, we will check that the required inequalities hold by direct calculation. To avoid the repetition of similar arguments, for the remaining families of simple groups of Lie type, we will provide a table to list all the exceptional cases that cannot be handled by the results in Section 3, and simply claim that one can check those by direct calculation or GAP ~\cite{GAP}.\\

\item Let $G$ be of type $A_1(q)$ where $q =p^f$. We have $|G| = q(q+1)(q-1)d^{-1}$ where $d=(2,q-1)$ and $|\Out(G)| = d f$.

Case a. Suppose that $q$ is even. Then $d=1$ and $|\Out(G)|=f$.

Assume there exists a Zsigmondy prime $L_1$ for $p^{2f}-1$ and there exists a Zsigmondy prime $L_2$ for $p^f-1$. Then clearly $L_1 \mid p^{2f}-1$ and thus $L_1 \mid p^{f}+1$ and $L_1 \geq 2f$. Also $L_2 \mid p^{f}-1$ and $L_2 \geq f$. It is easy to see that $L_1 \neq L_2$. So if we take subgroups $H_i$ of order $L_i$ ($i=1,2$), then we are done.

By Zsigmondy's Theorem, the exceptional cases are the following,
\begin{enumerate}[1)]
\item $p=2$ and $f=6$. Thus $|\Out(G)|=6$. Since $2^{6}-1=9\cdot7$, and $2^{6}+1=5\cdot13$, we can pick $|H_1|=7$ and $|H_2|=13$.
\item $p=2$ and $f=3$. Thus $|\Out(G)|=3$. Since $2^{3}-1=7$, and $2^3+1=9$, we can pick $|H_1|=7$ and $|H_2|=9$.
\end{enumerate}

Case b. Suppose that $q$ is odd. Then $d=2$ and $|\Out(G)| = 2f$.

Assume there exists a Zsigmondy prime $L_1$ for $p^{2f}-1$ and there exists a Zsigmondy prime $L_2$ for $p^f-1$. Then clearly $L_1 \mid p^{2f}-1$ and thus $L_1 \mid p^{f}+1$ and $L_1 \geq 2f$. Also $L_2 \mid p^{f}-1$ and $L_2 \geq f$. It is easy to see that $L_1 \neq L_2$.

If $L_2 \geq 2f$, then the result follows. By Lemma ~\ref{feitcorprime}, the exceptional cases are the following,
\begin{enumerate}[1)]
\item $f=1$, then $q=p \geq 3$, and $|\Out(G)|=2$. Since $4 \mid |G|$ and $p \mid |G|$, we can pick $|H_1|=4$ and $|H_2|=p$.
\item $f=2$, then $q=p^2 \geq 9$, and $|\Out(G)|=4$. Since $4 \mid |G|$ and $p^2 \mid |G|$, we can pick $|H_1|=4$ and $|H_2|=p^2$.
\item $p=3$, $f=4$, and $|\Out(G)| = 8$. Since $3^4-1=16\cdot5$, and $3^4+1=2\cdot41$, $|G|$ is divisible by $41$ and $9$. We can pick $|H_1|=9$ and $|H_2|=41$.
\end{enumerate}

\bigskip

\item Let $G$ be of type $A_n(q)$ where $q=p^f$ and $n \geq 2$. Set $m=\prod_{i=1}^{n} (q^{i+1}-1)$. Then $|G| = d^{-1} q^{n(n+1)/2} m$, $|\Out(G)| = 2fd$ where $d = (n+1, q-1)$. By ~\cite[Theorem 3.2]{simplegroups/book}, we know that $\Out(G) \cong D_{2d} \times C_f$ and it is not abelian if $d \geq 3$.

We observe that when $p=2$ and $f=1$, then $q=2$, $d=1$ and $|\Out(G)| = 2$. Since $\gcd(2^{n+1}-1, 2^{n}-1) \mid 3$, the result is easy to check. Thus we may assume that $q \neq 2$.

When $q=3$ and $f=1$, then $q=3$, $d \leq 2$ and $|\Out(G)| = 4$. Since $\gcd(3^{n+1}-1, 3^{n}-1) \mid 4$, $(3^{n+1}-1)(3^{n}-1)$ is divisible by $8$ and a prime $t\geq 5$, we can pick $|H_1|=4$ and $|H_2|=t$. Thus we may assume that $q \neq 3$.

By Lemma ~\ref{feitcor}, with the exception of a finite number of cases there exists a Zsigmondy prime $L_1$ for $p^{f(n+1)}-1$ such that $L_1 \geq 2f(n+1)$ or $L_1^2 \mid p^{f(n+1)}-1$ and hence $L_1^2 \geq 2 f(n+1)$. Moreover, by Lemma ~\ref{largercorprime}, with the exception of a finite number of cases there exists a prime $L_2 \mid p^{fn}-1$, such that $L_2 \geq 3fn \geq 2f(n+1)$, or $L_2^2 \mid p^{fn}-1$ and hence $L_2^2 \geq 3fn \geq 2f(n+1)$. It is easy to see that $L_1 \neq L_2$ and the result follows.

For all the exceptional cases, we can check the results through a direct calculation and by considering the order of the group $G$ and $\Out(G)$ (see the following table (Table 1)).

\Small
\begin{table}[ht]
\caption{} 
\centering

\begin{center}
\begin{tabular}{ c c c c c c c c c}
\hline
\hline
 $p$ & $n$ & $f$ &  $d$ & $|\Out(G)|$ & $|H_1|$ & $|H_2|$\\
\hline

 $p=q$ & 2 & 1 &  $(3,q-1)$ & $2d$ & $p^2$ & divides $(p^3-1)$\\
 2 & 2 & 2 & 3 & 12, $\Out(G)$ is nonabelian & 7 & 9\\
 2 & 2 & 3 & 1 & 6 & 7 & 73\\
 2 & 3 & 2 & 1 & 4 & 7 & 17\\
 2 & 2 & 4 & 3 & 24, $\Out(G)$ is nonabelian & 13 & 17\\
 2 & 4 & 2 & 1 & 4 & 17 & 31\\
 2 & 2 & 5 & 1 & 10 & 31 & 151\\
 2 & 5 & 2 & 3 & 12 & 13 & 31\\
 2 & 2 & 6 & 3 & 36, $\Out(G)$ is nonabelian & 19 & 73\\
 2 & 3 & 4 & 1 & 8 & 17 & 257\\
 2 & 4 & 3 & 1 & 6 & 31 & 151\\
 2 & 6 & 2 & 1 & 4 & 43 & 127\\
 2 & 2 & 9 & 1 & 18 & 19 & 73\\
 2 & 3 & 6 & 1 & 12 & 19 & 73\\
 2 & 6 & 3 & 7 & 42 & 73 & 127\\
 2 & 9 & 2 & 1 & 4 & 11 & 31\\
 2 & 2 & 10 & 3 & 60 & 151 & 331\\
 2 & 4 & 5 & 1 & 10 & 11 & 31\\
 2 & 5 & 4 & 1 & 8 & 11 & 31\\
 2 & 10 & 2 & 1 & 4 & 11 & 31\\
 3 & 2 & 2 &  1 & 4 & 5 & 7\\
 3 & 2 & 3 &  1 & 6 & 7 & 13\\
 3 & 3 & 2 &  4 & 16, $\Out(G)$ is nonabelian & 13 & 41\\
 5 & 2 & 3 &  1 & 6 & 19 & 31\\
 5 & 3 & 2 &  4 & 16, $\Out(G)$ is nonabelian & 31 & 313\\
 5 & 6 & 1 & 1 & 2 & 7 & 31\\
 2 & 3 & 3 & 1 & 6 & 7 & 13\\
\end{tabular}
\end{center}

\end{table}


\Large

\item  Let $G$ be of type $\leftidx{^2} A_n(q^2)$ where $n \geq 2$. Note that if $n = 2$, then $q > 2$. Set $m = \prod_{i=1}^{n} (q^{i+1}-(-1)^{i+1})$, $q^2=p^f$ and $d = (n + 1, q + 1)$. Then $|G| = d^{-1} m q^{n(n+1)/2}$, and $|\Out(G)| = d \cdot f $. By Lemma \ref{feitcor}, with the exception of a finite number of cases, there exists a Zsigmondy prime $L_1 \mid p^{f(n+1)/2}-(-1)^{n+1}$ such that $L_1 \geq (n+1)f$, or $L_1^2 \mid p^{f(n+1)/2}-(-1)^{n+1}$ which implies that $L_1^2 \geq (n+1)f$. Moreover, by Lemma \ref{largercorprime}, with the exception of a finite number of cases, there exists a Zsigmondy prime $L_2 \mid p^{fn/2}-(-1)^{n}$ such that $L_2 \geq (n+1)f$, or $L_2^2 \mid p^{fn/2}-(-1)^{n}$ which implies that $L_2^2 \geq 3 (fn/2) \geq (n+1)f$. For all the exceptional cases, we can check the results through a direct calculation and by considering the order of the group $G$ and $\Out(G)$ (see the following table (Table 2)).

\small
\begin{table}[ht]
\caption{} 
\centering
\begin{center}
\begin{tabular}{ c c c c c c c c c}
\hline
\hline
 $p$ & $n$ & $f/2$ &  $d$ & $|\Out(G)|$ & $|H_1|$ & $|H_2|$\\
\hline

$p=q$ & 2 & 1 & $(3, p+1)$ & $2d$ & $p^2$ & Zsig. prime of $p^6-1$\\
 2 & 3 & 1 & 1 & 2 & 5 & 9\\
 2 & 2 & 2 & 1 & 4 & 5 & 13\\
 2 & 4 & 1 & 1 & 2 & 5 & 11\\
 2 & 2 & 3 & 3 & 18, $\Out(G)$ is nonabelian by GAP & 9 & 19\\
 2 & 3 & 2 & 1 & 4 & 5 & 13\\
 2 & 6 & 1 & 1 & 2 & 7 & 43\\
 2 & 2 & 4 & 1 & 8 & 17 & 241\\
 2 & 4 & 2 & 5 & 20 & 25 & 41\\
 2 & 8 & 1 & 3 & 6 & 17 & 19\\
 2 & 2 & 5 & 3 & 30 & 31 & 331\\
 2 & 5 & 2 & 1 & 4 & 5 & 7\\
 2 & 10 & 1 & 1 & 2 & 11 & 31\\
 2 & 2 & 6 & 1 & 12 & 13 & 37\\
 2 & 3 & 4 & 1 & 8 & 17 & 241\\
 2 & 4 & 3 & 1 & 6 & 11 & 13\\
 2 & 6 & 2 & 1 & 4 & 5 & 7\\
 2 & 12 & 1 & 1 & 2 & 5 & 7\\
 2 & 2 & 9 & 3 & 54 & 73 & 87211\\
 2 & 3 & 6 & 1 & 12 & 13 & 37\\
 2 & 6 & 3 & 1 & 6 & 7 & 19\\
 2 & 9 & 2 & 5 & 20 & 31 & 41\\
 2 & 18 & 1 & 1 & 2 & 7 & 19\\
 2 & 2 & 10 & 1 & 20 & 41 & 61\\
 2 & 4 & 5 & 1 & 10 & 11 & 31\\
 2 & 5 & 4 & 1 & 8 & 9 & 13\\
 2 & 10 & 2 & 1 & 4 & 11 & 31\\
 2 & 20 & 1 & 3 & 12 & 31 & 41\\
 3 & 2 & 2 & 1 & 4 & 5 & 73\\
 3 & 4 & 1 & 1 & 2 & 5 & 61\\
 3 & 2 & 3 & 1 & 6 & 7 & 13\\
 3 & 3 & 2 & 2 & 8 & 41 & 73\\
 3 & 6 & 1 & 1 & 2 & 7 & 13\\
 5 & 2 & 3 & 3 & 18 & 31 & 5167\\
 5 & 3 & 2 & 2 & 8 & 13 & 313\\
 5 & 6 & 1 & 1 & 2 & 7 & 29\\
 2 & 5 & 1 & 3 & 6 & 7 & 11\\
 2 & 9 & 1 & 1 & 2 & 11 & 31\\
 2 & 3 & 3 & 1 & 6 & 7 & 13\\
 2 & 11 & 1 & 3 & 6 & 7 & 13\\
 2 & 5 & 3 & 3 & 18 & 19 & 73\\
 2 & 8 & 2 & 1 & 4 & 5 & 13\\
 2 & 17 & 1 & 3 & 6 & 7 & 19\\
 3 & 3 & 1 & 4 & 8, $\Out(G)$ is nonabelian by GAP & 5 & 7\\
 3 & 5 & 1 & 2 & 4 & 7 & 13\\
 5 & 2 & 2 & 1 & 4 & 13 & 601\\
 5 & 5 & 1 & 6 & 12 & 31 & 521\\
\end{tabular}
\end{center}

\end{table}

\Large

We now list the table (Table 3) for the simple groups of Lie types other than $A_n(q)$ and $\leftidx{^2} A_n(q^2)$. The second and the third columns in the table are the terms where we try to find two large coprime prime divisors.\\\\\\\\

\Small
\begin{table}[ht]
\caption{Other simple groups of Lie type} 
\centering 
\begin{tabular}{c c c c} 
\hline
\hline type & $L_1$ & $L_2$ & exceptional cases \\[0.5ex] 
\hline 
$B_n(q)$, $q=p^f$ & $p^{f(2n)}-1$ & $p^{f(2n-2)}-1$ & $(f=1,n=2);(p=2,f=3,n=2);(p=2,f=1,n=3)$. \\
$C_n(q)$, $q=p^f$ & $p^{f(2n)}-1$ & $p^{f(2n-2)}-1$ & $(f=1,n=2);(p=2,f=3,n=2);(p=2,f=1,n=3)$. \\
$D_n(q)$, $q=p^f$ & $p^{f(2n-2)}-1$ & $p^{f(2n-4)}-1$ & $(p=2,f=1,n=5);(p=2,f=2,n=5);(p=2,f=1,n=8)$; \\
 &   &  & $(p=2,f=1,n=4);(p=2,f=3,n=4)$; \\
 &   &  & $(p=3,f=1,n=4);(p=5,f=1,n=4)$. \\
$\leftidx{^2}D_n(q^2)$, $q^2=p^f$ & $p^{f(n-1)}-1$ & $p^{f(n-2)}-1$ & $(p=3,f=2,n=4);(p=5,f=2,n=4)$.\\
$E_6(q)$, $q=p^f$ & $p^{12 f}-1$ & $p^{8 f}-1$  & \\
$E_7(q)$, $q=p^f$ & $p^{18 f}-1$ & $p^{14 f}-1$  & \\
$E_8(q)$, $q=p^f$ & $p^{30 f}-1$ & $p^{24 f}-1$   & \\
$F_4(q)$, $q=p^f$ &  $p^{12 f}-1$ & $p^{8 f}-1$  & \\
$G_2(q)$, $q=p^f$  & $p^{6 f}-1$ & $p^{2 f}-1$ & $(p=2,f=3)$. \\
$\leftidx{^2}E_6(q^2)$, $q^2=p^f$ & $p^{6 f}-1$ & $p^{4 f}-1$ & \\
$\leftidx{^3} D_4(q^3)$, $q^3=p^f$ & $p^{4 f}-1$ & $p^{2 f}-1$ & $(p=2,f=3)$. \\
 $\leftidx{^2}B_2(2^{2n+1})$ & $2^{4(2n+1)}-1$ & $2^{2n+1}-1$ & \\
$\leftidx{^2}F_4(2^{2n+1})$ & $2^{4(2n+1)}-1$ & $2^{2n+1}-1$ & \\
$\leftidx{^2}G_2(3^{2n+1})$  & $3^{3(2n+1)}+1$ & $3^{2n+1}-1$ & \\
[1ex] 
\hline 
\end{tabular}
\label{table:primitivesmall} 
\end{table}
\end{enumerate}

\Large

Again, the exceptional cases can be checked by direct calculations and using the order of $G$. This completes the proof.
\end{proof}

The following two corollaries follow immediately from Proposition ~\ref{prop1}.

\begin{corollary}\label{prop3}
Let $G$ be a finite non-abelian simple group and $p$ be a fixed prime. Then there exists an abelian subgroup $H$ of $G$ such that $(|H|,p)=1$ and $2 \sqrt{|\Out(G)|} \leq |H|$.
\end{corollary}

\begin{corollary}\label{prop4}
Let $G$ be a finite non-abelian simple group and $p$ be a fixed prime. Then there exists an abelian subgroup $H$ of $G$ such that $(|H|,p)=1$ and for any subgroup $K$ of $\Out(G)$, ${|K/K'|} \leq |H|$.
\end{corollary}

\begin{lemma}\label{lem11}
Assume that $X= X_1 \times \dots \times X_m$ is a direct product of finite groups $X_i$ permuted transitively by $G$ where $m > 1$. If $V$ is a $G$-invariant subgroup of $X$, then $|V: [G,V]| \leq |X|^{1/2}$.
\end{lemma}
\begin{proof}
Choose $g \in G$ that leaves no $X_i$ invariant. Then $|C_X(g)| \leq |X|^{1/2}$. Hence $|V: [G,V]| \leq |V: [g,V]| \leq |C_V(g)| \leq |C_X(g)| \leq |X|^{1/2}$.
\end{proof}

\begin{proposition}\label{prop5}
Let $p$ be a fixed prime. Assume that $N=L_1 \times \dots \times L_m$ is a direct product of $m$ isomorphic finite non-abelian simple groups. Let $G$ be a group which acts faithfully via automorphisms on $N$ such that it permutes the $L_i$s  transitively. Then there exists an abelian subgroup $H=H_1 \times \dots \times H_m$ such that $H_i < L_i$, $(|H|,p)=1$ and $|G/G'| \leq |H|$.
\end{proposition}
\begin{proof}
If $m=1$, then $G$ is isomorphic to a subgroup of $\Aut(L_1)$. There exists an abelian group $H_1 < L_1$ such that $|G:G'| \leq |H_1|$ and $(|H_1|,p)=1$ by Corollary ~\ref{prop4}.

If $m>1$, we define $K=\bigcap_i \bN_G(L_i)$. Then $K$ is a $G$-invariant subgroup of $\Out L_1 \times \dots \times \Out L_m$. Thus $|K:[G,K]| \leq |\Out(L_1)|^{m/2}$ by Lemma ~\ref{lem11}. Thus $|G/G'|=|G:G'K||K:K\cap G'| \leq 2^{m-1}|\Out(L_1)|^{m/2}$ by ~\cite[Theorem 2]{aschbacher-guralnick}. There exists an abelian subgroup $H=H_1 \times \dots \times H_m$ such that $H_i < L_i$, $(|H|,p)=1$ and $|G/G'| \leq |H|$ by Corollary ~\ref{prop3}.
\end{proof}

\section{Main Theorem} \label{section5}

In this section we finally prove Theorem \ref{thm1}, which we restate for the reader's convenience.\\

{\bf Theorem 1.1} {\it Let $G$ be a finite group and $V$ a finite faithful completely reducible $G$-module, possibly of mixed characteristic. Let $M$ be the largest orbit size in the action of $G$ on $V$. Then
\[|G/G'|\leq M.\]
}
\begin{proof}
Suppose that $G$, $V$ is a counterexample such that $|GV|$ is minimal. First observe that with the same arguments used in the proof of \cite[Theorem 2.3]{keller-yang}
we may assume that $V$ is irreducible. Let $p$ be the characteristic of $V$. Then clearly $\text{O}_p(G)=1$.\\

Let $Z=Z(E(G))$. By the proof of \cite[Lemma 3.15]{gls}
and by \cite[Lemma 3.16]{gls} we know that $F(G/Z)=F(G)/Z$, $E(G/Z)=E(G)/Z$, $F^*(G/Z)=F^*(G)/Z$,
and $C_{G/Z}(E(G/Z))=C_G(E(G))/Z$.\\

We work in $\bar{G}=G/Z$ and put $\bar{C}=C_{\bar{G}}(E(\bar{G}))=C_G(E(G))/Z$.

By Corollary ~\ref{cor3} we have

\[|\bar{C} : \bar{C}'|\leq |F(\bar{G}):F(\bar{G})'|.\]

Since $E(G)/Z$ is a direct product of simple groups, we write $E(G)/Z=E_1 \times E_2 \times \cdots \times E_n$.

Consider $K=\bar G/ \bar C$, then $K$ acts faithfully on $E(G)/Z$. We may assume that $K_0=K$ acts transitively on $L_1=E_{11} \times E_{12} \times \cdots \times E_{1m_1}$. Let $K_1=C_{K}(L_1)$. We may assume that $K_1$ acts transitively on $L_2=E_{21} \times E_{22} \times \cdots \times E_{2m_2}$ and let $K_2=C_{K_1}(L_2)$. Inductively, we may define $L_3, K_3 \dots L_t, K_t$.

Clearly $E(G)/Z=L_1 \times L_2 \times \cdots \times L_t$ and we know that $K_{i-1}/K_i$ acts transitively and faithfully on $L_i$ where $i=1, \dots, t$.

By Proposition \ref{prop5}, there exists an abelian group $A_i$ where $A_i < L_i$, $(|A_i|,p)=1$ and $|K_{i-1}/K_i:(K_{i-1}/K_i)'| \leq |A_i|$ for $1 \leq i \leq t$.

Let $\bar A=A_1 \times \dots \times A_t$, thus $\bar A$ is abelian, $(|\bar A|,p)=1$ and $|K:K'| \leq |\bar A|$.

Hence altogether we have

\[|\bar{G}:\bar{G}'|\leq |K:K'| \ |\bar{C}:\bar{C}'|\leq
|\bar{A}| \ |F(\bar{G}):F(\bar{G})'|,\ \ \ \ \ (1)\]

and as $\bar{A} F(\bar{G}) = \bar{A} \times F(\bar{G})$, we see that for $\bar{T}=\bar{A} F(\bar{G})$ that

\[ |\bar{G}:\bar{G}'|\leq |\bar{T}:\bar{T}'|.\]

Now let $Z\leq T\leq G$ and $Z\leq A\leq G$ be the inverse images of $\bar{T}$ and $\bar{A}$, respectively.

Then $T=AF(G)$ is a central product of $A$ and $F(G)$ with common central subgroup $Z$; write
$T=A \Ydown_Z F(G)$. \\

Let $Z_1=A'\leq Z$, $Z_2=F(G)'\cap Z$, and $Z_0=Z_1\cap Z_2$. Then

\[T'= Z_1 \Ydown_{Z_0} F(G)'\]
and thus

\[|T:T'| = \frac{ \frac{|A|\ |F(G)|}{|Z|}}{\frac{|Z_1| \ |F(G)'|}{|Z_0|}}
        = |A/Z|\  |F(G):F(G)'|\ \frac{1}{|Z_1:Z_0|}.\]

Now observe that $Z_1\leq Z\cap G'$ and (by definition) $Z_2=Z\cap F(G)'$. Thus $Z_1Z_2\leq Z\cap G'$, and so

\[Z_1/Z_0\ =\ Z_1/(Z_1\cap Z_2)\ \cong\ Z_1Z_2/Z_2\ \leq\ (Z\cap G')/(Z\cap F(G)').\]

Therefore

\[|T:T'| \ \geq\ |A/Z|\  |F(G):F(G)'|\ \frac{|Z\cap F(G)'|}{|Z\cap G'|}.\ \ \ \ (2)\]

Now from (1) we get
\begin{eqnarray*}
|G:G'Z|& =&|(G/Z):(G/Z)'|\ =\ |\bar{G}:\bar{G}'|\\
&\leq &|\bar{A}|  |F(\bar{G}):F(\bar{G})'|\\
&= &|A/Z|  |(F(G)/Z):(F(G)/Z)'|\\
&= &|A/Z|  |F(G):F(G)'Z|\\
&= &|A/Z|  \frac{|F(G)| \ |Z\cap F(G)'|}{|F(G)'|\ |Z|}\\
&= &|A/Z|  |F(G):F(G)'|\frac{|Z\cap F(G)'|}{|Z|}.\ \ \ \ \ (3)\\
\end{eqnarray*}

Finally, combining (2) and (3) we have

\begin{eqnarray*}
|G:G'|& =&|G:G'Z|\ |G'Z:Z|\\
&= &|G:G'Z|\ |Z:(Z\cap G')|\\
&\leq  &|A/Z|  |F(G):F(G)'|\ \frac{|Z\cap F(G)'|}{|Z|}\ \frac{|Z|}{|Z\cap G'|}\\
&\leq & |T:T'|\\
\end{eqnarray*}
and $p$ does not divide $|T|$.

We note that $T$ is a solvable group that acts faithfully on $V$ and $(|T|,p)=1$. It is clear that $T$ acts completely reducibly on $V$. Thus $|G:G'| \leq |T:T'|\leq M$ by ~\cite[Theorem 1.1]{keller-yang}.
\end{proof}

\centerline{\bf Acknowledgements \rm}


The project was supported by NSFC (Nos: 11671063), a grant from the Simons Foundation (\#280770 to Thomas M. Keller), and a grant from the Simons Foundation (\#499532 to Yong Yang). The authors are also grateful to the referee for his/her valuable suggestions and comments.



\end{document}